\documentclass[english]{article}
\usepackage[T1]{fontenc}
\usepackage[latin9]{inputenc}
\usepackage{amsthm}
\usepackage{amsmath}
\usepackage{amssymb}
\usepackage{esint}

\makeatletter
\theoremstyle{plain}
\newtheorem{thm}{\protect\theoremname}
  \theoremstyle{definition}
  \newtheorem{defn}[thm]{\protect\definitionname}
  \theoremstyle{plain}
  \newtheorem{prop}[thm]{\protect\propositionname}
  \theoremstyle{remark}
  \newtheorem{rem}[thm]{\protect\remarkname}
  \theoremstyle{plain}
  \newtheorem{cor}[thm]{\protect\corollaryname}

\def\LyX{\texorpdfstring{%
  L\kern-.1667em\lower.25em\hbox{Y}\kern-.125emX\@}
  {LyX}}

\makeatother

\usepackage{babel}

\makeatother

\usepackage{babel}

\makeatother

\usepackage{babel}

\date{}

\makeatother

\usepackage{babel}
  \providecommand{\corollaryname}{Corollary}
  \providecommand{\definitionname}{Definition}
  \providecommand{\propositionname}{Proposition}
  \providecommand{\remarkname}{Remark}
\providecommand{\theoremname}{Theorem}

\begin{document}

\title{Holomorphic curves at one point}

\author{Erkao Bao}
\maketitle
\begin{abstract}
Let $(M,\omega)$ be a closed symplectic manifold with a compatible
almost complex structure $J.$ We prove that for $p\in M$ and $E>0$,
if $v:\Sigma\to M$ is a non-constant $J$-holomorphic curve with
symplectic area smaller than $E,$ then the number of the elements
in $v^{-1}(p)$ is bounded, and the bound is independent of $v$.
We also provide a uniform Hofer's energy bound for $J$-holomorphic
curves in $M\backslash p$ based on the symplectic area. Using these
two results we compactify the moduli space of $J$-holomorphic curves
in $M$ by adding holomorphic buildings at the point $p.$\\
\\

\end{abstract}
Let $(M,\omega)$ be a closed smooth symplectic manifold of dimension
$2N$ and $J$ be a compatible almost complex structure on $M$. For
a fixed a point $p\in M,$ we assume that $J$ is integrable inside
a small neighborhood $U$ of $p.$ For a sufficiently small neighborhood
$V\subseteq U$ of $p$ there exists a coordinate chart $\varphi:V\to B(0,\epsilon)\subseteq\mathbb{C}^{N}$
such that $\varphi(p)=0,$ $\varphi^{*}i=J$, and $\varphi^{*}\omega_{0}=\omega,$
where $B(0,\epsilon):=\left\{ \left.z\in\mathbb{C}^{N}\right||z|<\epsilon\right\} $
and $\omega_{0}:=\frac{i}{2}\sum_{k=1}^{n}dz_{k}\wedge d\bar{z}_{k}=\sum_{k=1}^{n}dx_{k}\wedge dy_{k}$
is the standard symplectic structure on $\mathbb{C}^{N}.$ We identify
$B(0,\epsilon)\backslash0$ with $W:=(-\infty,\log\epsilon)\times S^{2N-1}$
via the map $\psi(z)=(\log|z|,\frac{z}{|z|}).$ 

Let $\lambda$ be the standard contact $1$-form on the unit sphere
$S^{2N-1}\subseteq\mathbb{C}^{N},$ namely 
\[
\lambda=\left.\sum_{k=1}^{N}x_{k}dy_{k}-y_{k}dx_{k}\right|_{S^{2n-1}}.
\]
 The corresponding Reeb vector field on $S^{2N-1}$ is denoted by
$\mathbf{R},$ i.e. $\mathbf{R}\in\ker d\lambda$ and $\lambda(\mathbf{R})=1.$
We extend $\lambda$ trivially to $W$ via the pull back of the projection
map $\Theta:(-\infty,\log\epsilon)\times S^{2N-1}\to S^{2N-1}.$

Let $(\Sigma,j)$ be a Riemann surface with finitely many punctures
and $v:\mbox{\ensuremath{\Sigma}}\to M\backslash p$ be a $J$-holomorphic
curve, i.e. $J(u)\circ Tu=Tu\circ j$.
\begin{defn}
The Hofer's energy $E(v)$ of $v$ is defined to be 
\[
E(v)=E_{symp}(v)+E_{d\lambda}(v)+E_{\lambda}(v),
\]
\[
E_{symp}(v)=\int_{\Sigma}v^{*}\omega,
\]
\[
E_{d\lambda}(v)=\int_{v^{-1}(W)}v^{*}d\lambda,
\]
 and 
\[
E_{\lambda}(v)=\underset{\phi\in\mathcal{C}}{\sup}\int_{v^{-1}(W)}v^{*}(\phi(r)dr\wedge\lambda)
\]
 where 
\[
\mathcal{C}=\left\{ \phi\in C_{c}^{\infty}((-\infty,\log\epsilon),[0,1])\left|\int_{-\infty}^{\log\epsilon}\phi(x)dx=1\right.\right\} .
\]

\end{defn}
It is easy to check that $v^{*}\omega,$ $v^{*}d\lambda,$ and $v^{*}(\phi(r)dr\wedge\lambda)$
are non-negative multiples of a volume form on $\Sigma.$ 

Let $(r,\Theta)$ be the coordinate of $W:=(-\infty,\log\epsilon)\times S^{2N-1},$
$q$ be a puncture of $\Sigma,$ and $f$ be a biholomorphic map from
$(-\infty,0)\times\mathbb{R}/\mathbb{Z}$ to a small open subset of
$\Sigma$ around $q.$ If we choose $(s,t)$ as the coordinate of
$(-\infty,0)\times\mathbb{R}/\mathbb{Z},$ then the map $v\circ f$
can be written as $(r(s,t),\Theta(s,t)).$ We say the map $v$ converges
to a Reeb orbit of period $T>0$ around the puncture $q,$ if there
exists a map $\gamma:\mathbb{R}/\mathbb{Z}\to S^{2N-1}$ such that
\[
\frac{d}{dt}\left(\gamma(Tt)\right)=\mathbf{R}\left(\gamma(Tt)\right),
\]
\[
\begin{array}{ccccc}
\underset{s\to-\infty}{\lim}\Theta(s,t)=\gamma(Tt), &  & \mbox{and} &  & \underset{s\to-\infty}{\lim}\frac{r(s,t)}{s}=T.\end{array}
\]
 The definition of converging to a Reeb orbit is independent of the
choice of $f.$ It is easy to see that in this case $T=2k\pi$ for
some $k\in\mathbb{Z}_{>0}.$ Every Reeb orbit of $\mathbf{R}$ satisfies
the Morse-Bott condition, so the result in \cite{morse bott} applied
to this special case gives us
\begin{thm}
\label{thm:converging to Reeb} \cite{morse bott} Assume the $J$-holomorphic
curve $v$ has finite Hofer's energy, i.e. $E(v)<+\infty.$ Around
each puncture of $\Sigma,$ v converges to either a point in $M\backslash p$
or a Reeb orbit. 
\end{thm}

We say a puncture $q$ is removable if around $q,$ $v$ converges
to a point in $M\backslash p.$ Otherwise, we say $q$ is non-removable.
The next proposition says that the number of non-removable punctures
of $v$ is bounded by a finite constant independent of $v.$

\begin{prop}
\label{prop:number of negative ends}Given $E>0,$ there exists a
number $N\in\mathbb{N}$ such that for any finitely punctured Riemann
surface $(\Sigma,j),$ and any non-constant $J$-holomorphic map $v:\Sigma\to M\backslash p$
with $E(v)\leqq E,$ the number of non-removable punctures of $v$
is no greater than $N.$\end{prop}
\begin{proof}
Suppose to the contrary. Let $v_{n}$ be a $J$-holomorphic curves
from a finitely punctured Riemann surface $\Sigma_{n}$ to $M\backslash p$
with $E(v_{n})\leqq E,$ such that the number of non-removable punctures
of $v_{n}$ goes to infinity as $n\to\infty.$ 

Let $\eta^{q_{n}}$ be the Reeb orbit of $\mathbf{R}$ to which around
the non-removable puncture $q_{n}$ , $v_{n}$ converges. We denote
the period of $\eta^{q_{n}}$ by $2k_{q_{n}}\pi,$ with $k_{q_{n}}\in\mathbb{Z}_{>0}.$
We can pick $[r_{n}-1,r_{n}+1]\times S^{2N-1}$ inside $W=(-\infty,\log\epsilon)\times S^{2N-1},$
such that $v_{n}^{-1}\left\{ [r_{n}-1,r_{n}+1]\times S^{2N-1}\right\} $
consists of connected components $A_{l,n}\subset\Sigma_{n}$ for $l\in I_{n}$
with 
\[
\partial A_{l,n}=\partial_{1}A_{l,n}-\partial_{2}A_{l,n},
\]
\[
v_{n}(\partial_{1}A_{l,n})\subset\{r_{n}+1\}\times S^{2N-1},
\]
\[
v_{n}(\partial_{2}A_{l,n})\subset\{r_{n}-1\}\times S^{2N-1},
\]
 and the cardinality $|I_{n}|$ of the index set $I_{n}$ satisfies
$|I_{n}|=\sum k_{q_{n}},$ where the summation is taken over all the
punctures $q_{n}$ for fixed $n.$ Thus, $|I_{n}|$ goes to $\infty,$
as $n\to\infty.$

Pick $\phi_{n}(r)\in\mathcal{C}$ to be a function satisfying $\phi(r)=\frac{1}{3}$
for $r_{n}-1\leqq r\leqq r_{n}+1.$ Look at the 2-form $\Omega_{n}:=\phi_{n}(r)dr\wedge\lambda+d\lambda$
on $W$, we know it is non-degenerate over $[r_{n}-1,r_{n}+1]\times S^{2N-1}.$
Since $S^{2N-1}$ is compact, by the Gromov's Monotonicity Theorem
we have $\int_{A_{l,n}}v_{n}^{*}\Omega_{n}>\delta_{0}>0,$ for some
$\delta_{0}$ independent of $n$. Therefore, we get 
\begin{eqnarray*}
E & \geqq & E(v_{n})\\
 & \geqq & E_{d\lambda}(v_{n})+E_{\lambda}(v_{n})\\
 & = & \int_{v_{n}^{-1}(W)}v_{n}^{*}d\lambda+\underset{\phi\in\mathcal{C}}{\sup}\int_{v_{n}^{-1}(W)}v_{n}^{*}(\phi(r)dr\wedge\lambda)\\
 & \geqq & \int_{v_{n}^{-1}(W)}v_{n}^{*}d\lambda+\int_{v_{n}^{-1}(W)}v_{n}^{*}(\phi_{n}(r)dr\wedge\lambda)\\
 & \geqq & \int_{v_{n}^{-1}([r_{n}-1,r_{n}+1]\times S^{2N-1})}v_{n}^{*}\Omega_{n}\to\infty.\\
 & = & \sum_{l\in I_{n}}\int_{A_{l,n}}v_{n}^{*}\Omega_{n}\\
 & \geqq & \sum_{l\in I_{n}}\delta_{0}\\
 & = & \delta_{0}|I_{n}|\to+\infty.
\end{eqnarray*}
\end{proof}
\begin{rem}
For each non-removable punctures of $v,$ we can associate a multiplicity
which is the multiplicity of the Reeb orbit. The above proof actually
shows that the number of non-removable punctures of $v$ counted with
multiplicity is bounded by $N.$
\end{rem}
For a non-constant $J$-holomorphic curve $v$ from a Riemann surface
$(S,j)$ to $M,$ we know $v^{-1}(p)$ is discrete, and hence finite.
Let $(\Sigma,j)$ be the punctured Riemann surface defined by $(S\backslash v^{-1}(p),j).$
Now $v$ can be viewed as a $J$-holomorphic curve from $\Sigma$
to $M\backslash p.$ From the local behavior of a holomorphic map,
we know that $v$ converges to a Reeb orbit along each puncture of
$\Sigma.$ From this we can see by the Stokes' Theorem that $E(v|_{\Sigma})<\infty.$
We define the Hofer's energy $E(v)$ of $v$ to be $E(v|_{\Sigma}).$
The following proposition says the number of elements in $v^{-1}(p)$,
denoted by $\left|v^{-1}(p)\right|$ is bounded by a finite number
depends on the Hofer's energy $E(v),$ and independent of $v.$ 
\begin{prop}
\label{prop:number of 0 by Hofer} Given $E>0$, there exists $N$
such that for any Riemann surface $(S,j)$, and any non-constant $J$-holomorphic
curve $v:S\to M$ with $E(v)\leqq E,$ we have $\left|v^{-1}(p)\right|\leqq N.$ \end{prop}
\begin{proof}
It follows directly from the Proposition \ref{prop:number of negative ends}.\end{proof}
\begin{rem}
The Proposition \ref{prop:number of 0 by Hofer} is also true if we
count multiplicity.
\end{rem}

The following theorem gives a Hofer's energy bound by the Symplectic
area (compare to 9.2 in \cite{compactness}). 
\begin{thm}
\label{thm: hofer's energy bd by symplectic area} There exists $C_{1}>0,$
such that for any Riemann surface $(S,j)$ and any non-constant $J$-holomorphic
curve $v:S\to M$ satisfying $E_{symp}(v)<+\infty,$ we have $E(v)\leqq C_{1}E_{symp}(v).$ \end{thm}
\begin{proof}
We restrict $v$ to the punctured Riemann surface $\Sigma:=S\backslash v^{-1}(p),$
so around each puncture $q\in v^{-1}(p),$ $v$ converges to a Reeb
orbit. Pick $\varepsilon\in\left[\frac{1}{2}\log\epsilon,\frac{2}{3}\log\epsilon\right]$
such that $\varepsilon$ is a regular value of $r\circ v,$ where
$r:W\to(-\infty,\log\epsilon)$ is the projection map defined by $(r,\Theta)\mapsto r.$
Denote $A:=v^{-1}((-\infty,\varepsilon]\times S^{2n-1})\subseteq\Sigma$
and $B_{+}:=v^{-1}(\{\varepsilon\}\times S^{2N-1})$. Let $\hat{A}$
be the orient blow up of $A$ around all $q$'s in $v^{-1}(p)$, i.e.
$\hat{A}=A\sqcup B_{-}$ with $B_{-}:=\underset{q\in v^{-1}(p)}{\bigsqcup}S_{q}^{1}$
being the disjoint union of circles parametrized by $v^{-1}(p).$
Hence we have $\partial\hat{A}=B_{+}-B_{-}.$ We continuously extend
$v$ to $\hat{A}$ by defining $v|_{S_{q}^{1}}$ to be the Reeb orbit
at negative infinity to which $v$ converges around $q.$ Now we have
\begin{eqnarray}
E_{d\lambda}(v) & = & \int_{\hat{A}}v^{*}d\lambda+\int_{\{\Sigma\backslash A\}\cap v^{-1}(W)}v^{*}d\lambda\nonumber \\
 & \leqq & \int_{B_{+}}v^{*}\lambda-\int_{B_{-}}v^{*}\lambda+\int_{\{\Sigma\backslash A\}\cap v^{-1}(W)}v^{*}d\lambda\nonumber \\
 & \leqq & \int_{B_{+}}v^{*}\lambda+\int_{\{\Sigma\backslash A\}\cap v^{-1}(W)}v^{*}d\lambda.\label{eq:d lambda}
\end{eqnarray}

Given $\phi\in\mathcal{C}=\left\{ \phi\in C_{c}^{\infty}((-\infty,\log\epsilon),[0,1])\left|\int_{-\infty}^{\log\epsilon}\phi(x)dx=1\right.\right\} ,$
we define $\Phi(r):=\int_{-\infty}^{r}\phi(x)dx,$ so we get 

\begin{eqnarray}
 &  & \int_{v^{-1}(W)}v^{*}(\phi(r)dr\wedge\lambda)\nonumber \\
 & = & \int_{\hat{A}}v^{*}(\phi(r)dr\wedge\lambda)+\int_{\{\Sigma\backslash A\}\cap v^{-1}(W)}v^{*}(\phi(r)dr\wedge\lambda)\nonumber \\
 & = & \int_{\hat{A}}v^{*}d(\Phi(r)\lambda)-\int_{\hat{A}}v^{*}(\Phi(r)d\lambda)+\int_{\{\Sigma\backslash A\}\cap v^{-1}(W)}v^{*}(\phi(r)dr\wedge\lambda)\nonumber \\
 & \leqq & \int_{\hat{A}}v^{*}d(\Phi(r)\lambda)+\int_{\{\Sigma\backslash A\}\cap v^{-1}(W)}v^{*}(dr\wedge\lambda)\nonumber \\
 & = & \int_{B_{+}}v^{*}(\Phi(r)\lambda)-\int_{B_{-}}v^{*}(\Phi(r)\lambda)+\int_{\{\Sigma\backslash A\}\cap v^{-1}(W)}v^{*}(dr\wedge\lambda)\nonumber \\
 & \leqq & \int_{B_{+}}v^{*}\lambda+\int_{\{\Sigma\backslash A\}\cap v^{-1}(W)}v^{*}(dr\wedge\lambda)\label{eq:lambda}
\end{eqnarray}

We choose a smooth function $\tau$ defined by $\tau(r)=\frac{\log\epsilon-r}{\log\epsilon-\varepsilon}$
for $\varepsilon\leqq r\leqq\log\epsilon.$ Since $\tau(\varepsilon)=1$
and $\tau(\log\epsilon)=0,$ by Stokes' Theorem we have
\begin{equation}
\int_{B_{+}}v^{*}\lambda=\int_{\{\Sigma\backslash A\}\cap v^{-1}(W)}v^{*}d\left(\tau(r)\lambda\right).\label{eq:boundary B+}
\end{equation}

Since the symplectic form $\omega$ is non-degenerate, on $\{\Sigma\backslash A\}\cap v^{-1}(W)$
we have $v^{*}d\lambda\leqq C(\epsilon)v^{*}\omega$ and $v^{*}d\left(\tau(r)\lambda\right)\leqq C(\epsilon)v^{*}\omega.$
Therefore from ($\ref{eq:d lambda}$), (\ref{eq:lambda}), and (\ref{eq:boundary B+})
we get 

\[
E(v)\leqq C_{1}(\epsilon)E_{symp}(v).
\]

\end{proof}
Proposition \ref{prop:number of 0 by Hofer} and Theorem \ref{thm: hofer's energy bd by symplectic area}
imply the following theorem
\begin{thm}
\label{thm:number bdd by symp area} Given $E>0$ there exists $N\in\mathbb{N},$
such that for any Riemann surface $(S,j)$ and any non-constant $J$-holomorphic
curve $v:\Sigma\to M$ with symplectic area $E_{symp}(v)\leqq E$,
we have $\left|u^{-1}(p)\right|\leqq N.$ 
\end{thm}

Let $\mathcal{M}_{g}(M,J,p,n,E)$ be the space of equivalence classes
of $(v,\Sigma,j,\vec{z})$ such that $(\Sigma,j)$ is a smooth Riemann
surface of genus $g;$ $\vec{z}=(z_{1},z_{2},...,z_{n})$ are $n\in\mathbb{N}$
distinct points in $\Sigma;$ $v:\Sigma\to M$ is a $J$-holomorphic
curve satisfying $v^{-1}(p)=\{z_{1},z_{2},...,z_{n}\};$ $E_{symp}(v)\leqq E;$
if $v$ is a constant map, we require $v\neq p$ and $g\geqq2.$ $(v,\Sigma,j,\vec{z})$
and $(v',\Sigma',j',\vec{z}')$ are equivalent if there exists a diffeomorphism
$\sigma:\Sigma\to\Sigma'$ satisfying $\sigma^{*}j'=j,$ $\sigma(z_{i})=z_{i}'$
for $i=1,2,...,n,$ and $v=v'\circ\sigma.$ 

Let $\mathcal{M}_{g}^{SFT}(M\backslash p,J,n,E)$ be the space of
equivalent classes of $(v,\Sigma,j,\vec{z})$ such that $(\Sigma,j)$
is a smooth Riemann surface of genus $g;$ $\vec{z}=(z_{1},z_{2},...,z_{n})$
are $n$ distinct points (called punctures) in $\Sigma;$ $v:(\Sigma\backslash\{z_{1},z_{2},...,z_{n}\},j)\to(M\backslash p,J)$
is a $J$-holomorphic curve such that around each puncture $z_{i}$
for $i=1,2,...,n,$ $v$ converges to a Reeb orbit of $S^{2N-1};$
$E(v)\leqq E;$ if $v$ is constant, we require $g\geqq2.$ $(v,\Sigma,j,\vec{z})$
and $(v',\Sigma',j',\vec{z}')$ are equivalent if there exists a diffeomorphism
$\sigma:\Sigma\to\Sigma'$ satisfying $\sigma^{*}j'=j,$ $\sigma(z_{i})=z_{i}'$
for $i=1,2,...,n,$ and $v=v'\circ\sigma.$ 

Proposition \ref{prop:number of negative ends} and Proposition \ref{prop:number of 0 by Hofer}
imply 
\begin{cor}
\label{cor:vanish} Given $E_{1}>0,$ there exists $N_{1}\in\mathbb{N},$
such that for $n_{1}>N_{1},$ $\mathcal{M}_{g}(M,J,p,n_{1},E_{1})=\emptyset.$
Given $E_{2}>0,$ there exists $N_{2}\in\mathbb{N},$ such that for
$n_{2}>N_{2},$ $\mathcal{M}_{g}^{SFT}(M\backslash p,J,n_{2},E_{2})=\emptyset.$
\end{cor}
There exists an obvious map $\pi$ from 
\[
\mathcal{M}_{g}^{SFT}(M\backslash p,J,n):=\underset{E\geqq0}{\bigcup}\mathcal{M}_{g}^{SFT}(M\backslash p,J,n,E)
\]
to 
\[
\mathcal{M}_{g}(M,J,p,n):=\underset{E\geqq0}{\bigcup}\mathcal{M}_{g}(M,J,p,n,E)
\]
by $\pi([v,\Sigma,j,\vec{z}])=[v,\Sigma,j,\vec{z}]$. $\pi$ is a
well defined bijective map by the Removable of Singularity Theorem
and the local behavior of holomorphic maps. 

We equip $\mathcal{M}_{g}(M,J,p,n,E)$ with the standard Gromov's
topology, and equip $\mathcal{M}_{g}^{SFT}(M\backslash p,J,n,E)$
with the Symplectic Field Theory topology introduced in \cite{compactness}.
From \cite{compactness} we know that $\pi$ is a homeomorphism.

To compactify $\mathcal{M}_{g}(M,J,p,n,E),$ we look at the moduli
space of curves with point-wise constraints, denoted by $\mathcal{M}_{g}^{*}(M,J,p,n,E),$
and defined by replacing the requirement ``$v^{-1}(p)=\{z_{1},z_{2},...,z_{n}\}$''
in the definition of $\mathcal{M}_{g}(M,J,p,n,E)$ by the requirement
``$v(z_{i})=p$ for $i=1,2,...,n$''. There exists an obvious inclusion
map 
\[
i:\mathcal{M}_{g}(M,J,p,n,E)\to\mathcal{M}_{g}^{*}(M,J,p,n,E).
\]
 We denote by $\overline{\mathcal{M}_{g}^{*}(M,J,p,n,E)}$ the standard
Gromov's compactification of $\mathcal{M}_{g}^{*}(M,J,p,n,E).$ We
define $\overline{\mathcal{M}_{g}(M,J,p,n,E)}$ to be the closure
of the subset $i\left(\mathcal{M}_{g}(M,J,p,n,E)\right)$ in $\overline{\mathcal{M}_{g}^{*}(M,J,p,n,E)},$
and hence $\overline{\mathcal{M}_{g}(M,J,p,n,E)}$ is compact. 

Denote by $\overline{\mathcal{M}_{g}^{SFT}(M\backslash p,J,n,E)}$
the Symplectic Field Theory compactification of $\mathcal{M}_{g}^{SFT}(M\backslash p,J,n)$
by adding holomorphic buildings (see \cite{compactness}). 
\begin{thm}
\cite{compactness} $\overline{\mathcal{M}_{g}^{SFT}(M\backslash p,J,n,E)}$
is compact.
\end{thm}
Denote $\overline{\mathcal{M}_{g}(M,J,p,n)}:=\underset{E\geqq0}{\bigcup}\overline{\mathcal{M}_{g}(M,J,p,n,E)}$
and $\overline{\mathcal{M}_{g}^{SFT}(M\backslash p,J,n)}:=\underset{E\geqq0}{\bigcup}\overline{\mathcal{M}_{g}^{SFT}(M\backslash p,J,n,E)}$.
The map $\pi$ extends continuously to a surjective map from $\overline{\mathcal{M}_{g}^{SFT}(M\backslash p,J,n)}$
to $\overline{\mathcal{M}_{g}(M,J,p,n)}.$ Actually, for $[v]\in\overline{\mathcal{M}_{g}^{SFT}(M\backslash p,J,n)},$
$\pi$ is just the forgetful map defined by forgetting the all the
negative levels, if any, of the holomorphic building $[v]$. 

Define $\overline{\mathcal{M}_{g}(M,J,p,n,E)}^{SFT}$ to be the closure
of the subset $\pi^{-1}\left(\mathcal{M}_{g}(M,J,p,n,E)\right)$ in
$\overline{\mathcal{M}_{g}^{SFT}(M\backslash p,J,n)}.$ Theorem \ref{thm: hofer's energy bd by symplectic area}
implies that there exists $E_{2}<\infty$ such that $\overline{\mathcal{M}_{g}(M,J,p,n,E)}^{SFT}\subseteq\overline{\mathcal{M}_{g}^{SFT}(M\backslash p,J,n,E_{2})},$
and hence 
\begin{cor}
$\overline{\mathcal{M}_{g}(M,J,p,n,E)}^{SFT}$ is compact.
\end{cor}
Now $\pi$ restricts to a map from $\overline{\mathcal{M}_{g}(M,J,p,n,E)}^{SFT}$
to $\overline{\mathcal{M}_{g}(M,J,p,n,E)}.$ Since for any $[v]\in\mathcal{M}_{g}(M,J,p,n,E),$
$\pi^{-1}\left([v]\right)$ consists of exactly one point, $\overline{\mathcal{M}_{g}(M,J,p,n,E)}^{SFT}$
can be viewed as another compactification of $\mathcal{M}_{g}(M,J,p,n,E).$
If we take a further quotient which makes the $n$ punctures unordered,
we get two corresponding spaces and denote them by $\overline{\mathcal{M}_{g}(M,J,p,n,E)_{\sharp}}^{SFT}$
and $\overline{\mathcal{M}_{g}(M,J,p,n,E)_{\sharp}}.$ Let $\mathcal{M}_{g}(M,J,E)$
be the standard moduli space of $J$-holomorphic curves of genus $g$
in $M,$ with symplectic area no greater than $E,$ and let $\overline{\mathcal{M}_{g}(M,J,E)}$
be the Gromov's compactification of $\mathcal{M}_{g}(M,J,E).$ We
have an obvious homeomorphism $h$ from $\underset{n}{\bigcup}\overline{\mathcal{M}_{g}(M,J,p,n,E)_{\sharp}}$
to $\overline{\mathcal{M}_{g}(M,J,E)},$ and a continuos map $\pi'=h\circ\pi$
from $\overline{\mathcal{M}_{g}(M,J,E)}^{p}:=\underset{n}{\bigcup}\overline{\mathcal{M}_{g}(M,J,p,n,E)_{\sharp}}^{SFT}$
to $\overline{\mathcal{M}_{g}(M,J,E)}.$ The Corollary \ref{cor:vanish}
implies 
\begin{cor}
\textup{$\overline{\mathcal{M}_{g}(M,J,E)}^{p}$ is compact.}

\textup{This means that $\overline{\mathcal{M}_{g}(M,J,E)}^{p}$ serves
as a compactification of $\mathcal{M}_{g}(M,J,E).$ }\end{cor}

\end{document}